\batchmode
\documentclass[12pt,english]{article}
\usepackage[T1]{fontenc}
\usepackage[latin9]{inputenc}
\usepackage{geometry}
\usepackage{babel}
\usepackage{array}
\usepackage{longtable}
\usepackage{rotfloat}
\usepackage{url}
\usepackage{amsthm}
\usepackage{amsmath}
\usepackage{amssymb}
\usepackage[unicode=true,pdfusetitle,
 bookmarks=true,bookmarksnumbered=false,bookmarksopen=false,
 breaklinks=false,pdfborder={0 0 1},backref=false,colorlinks=false]
 {hyperref}
\usepackage{breakurl}

\makeatletter

\providecommand{\tabularnewline}{\\}

\numberwithin{equation}{section}
\numberwithin{figure}{section}
\theoremstyle{plain}
\newtheorem{thm}{\protect\theoremname}[section]
  \theoremstyle{definition}
  \newtheorem{example}[thm]{\protect\examplename}
  \theoremstyle{plain}
  \newtheorem{lem}[thm]{\protect\lemmaname}
  \theoremstyle{remark}
  \newtheorem{rem}[thm]{\protect\remarkname}
  \theoremstyle{definition}
  \newtheorem{defn}[thm]{\protect\definitionname}
  \theoremstyle{plain}
  \newtheorem{prop}[thm]{\protect\propositionname}
  \theoremstyle{plain}
  \newtheorem{question}[thm]{\protect\questionname}
  \theoremstyle{plain}
  \newtheorem{cor}[thm]{\protect\corollaryname}

\usepackage{ae}
\usepackage{aecompl}

  \newlength{\BiblioSpacing}
  \renewenvironment{thebibliography}[1]{%
    \begin{oldthebibliography}{#1}%
      \setlength{\parskip}{\BiblioSpacing}
      \setlength{\itemsep}{\BiblioSpacing}
  }%
  {%
    \end{oldthebibliography}%
  }

\makeatother

  \providecommand{\corollaryname}{Corollary}
  \providecommand{\definitionname}{Definition}
  \providecommand{\examplename}{Example}
  \providecommand{\lemmaname}{Lemma}
  \providecommand{\propositionname}{Proposition}
  \providecommand{\questionname}{Question}
  \providecommand{\remarkname}{Remark}
\providecommand{\theoremname}{Theorem}

\begin{document}
\global\long\def\st{\operatorname{star}}

\global\long\def\link{\operatorname{link}}

\global\long\def\disjointunion{\mathbin{\dot{\cup}}}

\global\long\def\alexdual#1{#1^{\vee}}

\title{Chordal and sequentially Cohen-Macaulay clutters}

\author{Russ Woodroofe\\
{\small Department of Mathematics}\\
{\small Washington University in St. Louis}\\
{\small St. Louis, MO, 63130}\\
\texttt{\small russw@math.wustl.edu}}

\date{
}
\maketitle
\begin{abstract}
We extend the definition of chordal from graphs to clutters. The resulting
family generalizes both chordal graphs and matroids, and obeys many
of the same algebraic and geometric properties. Specifically, the
independence complex of a chordal clutter is shellable, hence sequentially
Cohen-Macaulay; and the circuit ideal of a certain complement to such
a clutter has a linear resolution. Minimal non-chordal clutters are
also closely related to obstructions to shellability, and we give
some general families of such obstructions, together with a classification
by computation of all obstructions to shellability on 6 vertices. 
\end{abstract}

\section{\label{sec:Introduction}Introduction}

A \emph{clutter} $\mathcal{C}$ is a hypergraph such that no edge
of $\mathcal{C}$ is properly contained in any other edge. For example,
any graph is a clutter, as is any $d$-uniform hypergraph. There is
a dual relationship between simplicial complexes and clutters, as
follows: Given any clutter $\mathcal{C}$, there is an \emph{independence
complex} $I(\mathcal{C})$ which has faces consisting of all subsets
of $V(\mathcal{C})$ containing no edge of $\mathcal{C}$. Given any
simplicial complex $\Delta$, there is a \emph{non-face clutter} $\mathcal{C}(\Delta)$
on the same vertex set with edges consisting of the minimal subsets
of $V(\Delta)$ which are not faces. Clearly $I(\mathcal{C}(\Delta))=\Delta$
and $\mathcal{C}(I(\mathcal{C}))=\mathcal{C}$.

The non-face clutter of $\Delta$ is perhaps most familiar via the
\emph{Stanley-Reisner ring} of $\Delta$:
\[
R[\Delta]\triangleq R[x_{1},\dots,x_{n}]/\left(x_{i_{1}}x_{i_{2}}\cdots x_{i_{k}}\,:\,\{x_{i_{1}},\dots,x_{i_{k}}\}\mbox{ an edge of }\mathcal{C}(\Delta)\right),
\]
where $V(\Delta)=\{x_{1},\dots,x_{n}\}$. The ideal in this construction
is known as the \emph{edge ideal} or \emph{circuit ideal} of $\mathcal{C}(\Delta)$.

Recently, a number of papers \cite{Dochtermann/Engstrom:2009,Ehrenborg/Hetyei:2006,Francisco/VanTuyl:2007,Ha/VanTuyl:2008,VanTuyl/Villarreal:2008,Woodroofe:2009a}
have asked what one can say about the algebraic and topological combinatorics
of $\Delta$ from the structure of $\mathcal{C}=\mathcal{C}(\Delta)$.
A particularly successful case has been that where $\mathcal{C}=G$
is a chordal graph. In this case, the independence complex $I(G)$
is vertex decomposable \cite{Dochtermann/Engstrom:2009,Woodroofe:2009a},
hence shellable \cite{VanTuyl/Villarreal:2008} and sequentially Cohen-Macaulay
\cite{Francisco/VanTuyl:2007}, while the edge ideal of the complement
$\overline{G}$ has a linear resolution \cite{Froberg:1990}. Moreover,
the chordal graphs are closely related to the largest family of graphs
having independence complexes such that every induced subcomplex is
shellable \cite{Woodroofe:2009a}; and the complements of chordal
graphs are exactly the graphs with edge ideal having a linear resolution
\cite{Froberg:1990}.

In the current paper, our purpose is to introduce a family of clutters,
which we call \emph{chordal clutters}, which satisfy similar properties.
Chordal clutters generalize several previously studied families, including
chordal graphs, circuit clutters of matroids, and acyclic clutters.
We will prove:
\begin{thm}
\label{thm:ChordalClutterShellable}If $\mathcal{C}$ is a chordal
clutter then the independence complex $I(\mathcal{C})$ is shellable
and hence sequentially Cohen-Macaulay.
\end{thm}
In particular, we obtain a uniform proof of shellability for independence
complexes of both chordal graphs and matroids.

We also prove:
\begin{thm}
\label{thm:ClutterLinearRes}Let $\mathcal{C}$ be a $d$-uniform
chordal clutter. Then the circuit ideal of the complement $d$-uniform
clutter of $\mathcal{C}$ has a linear resolution.
\end{thm}
As previously mentioned, there is a converse to Theorem \ref{thm:ClutterLinearRes}
and an interesting partial converse to Theorem \ref{thm:ChordalClutterShellable}
in the case where $\mathcal{C}$ is a graph. We discuss the possibility
of finding such converse results for general chordal clutters. We
relate Theorem \ref{thm:ChordalClutterShellable} to obstructions
to shellability via both examples and computational results, but conclude
that non-trivial description of circuit ideals having linear resolution
or linear quotients (i.e., a converse to Theorem \ref{thm:ClutterLinearRes})
is unlikely. 

\medskip{}

The structure of the papers is as follows. In Section \ref{sec:Background}
we present the background material. In Section \ref{sec:k-decomposableComplexes}
we define $k$-decomposability for non-pure simplicial complexes,
and extend many theorems proved by Provan and Billera \cite{Provan/Billera:1980}
for pure complexes. In Section \ref{sec:ChordalClutters} we define
simplicial vertices in clutters, which naturally leads us to define
chordal clutters. We give several examples, including chordal graphs
and the circuit clutters of matroids. In Section \ref{sec:Shellability}
we prove that the independence complex of any chordal clutter is shellable,
as a special case of a more general result about shellability of independence
complexes. In Section \ref{sec:LinearResolutions} we recall the basic
facts about linear resolutions, combinatorial Alexander duality, and
their relationship. We then use the Alexander dual to prove that the
circuit ideal of a certain uniform complement of a chordal clutter
has a linear resolution, and indeed has linear quotients. We close
in Section \ref{sec:ForbiddenMinors} by relating forbidden minors
to chordality with obstructions to shellability. We give several infinite
families of these forbidden minors, and characterize by computation
with GAP \cite{GAP4.4.12} all obstructions to shellability on $6$
or fewer vertices that contain no non-shellable link.

\subsection{Notation}

We will use letters $\mathcal{C}$ and $\mathcal{D}$ for clutters,
except in the special case where the clutter is a graph, when we defer
to convention and use $G$. Other calligraphic letters such as $\mathcal{F}$
will denote families of objects. Vertices in clutters will be denoted
by $v$, $w$, etc; while circuits (edges) will be denoted with the
letter $e$. We use upper case Greek letters such as $\Delta$ and
$\Sigma$ for simplicial complexes. Inside a simplicial complex, we
use letters $v$ and $w$ for vertices, and lower case Greek letters
such as $\sigma$ and $\tau$ for faces. \vspace{-0.1in}

\subsection*{Acknowledgements}

Eric Emtander explained several aspects of his (alternative) definition
of chordal clutter. Volkmar Welker pointed out the connection between
$c$-obstructions and Buchsbaum complexes. I have benefited greatly
from the interest and advice of John Shareshian.

\section{\label{sec:Background}Background}

\subsection{Simplicial complexes and clutters}

An \emph{abstract simplicial complex} $\Delta$ on a vertex set $V$
is a set of subsets of $V$ (called \emph{faces} of $\Delta$) such
that each subset of a face of $\Delta$ is itself a face of $\Delta$.
We do not require that singleton subsets (i.e. elements) of $V$ are
faces of $\Delta$. A maximal face is called a \emph{facet}, and a
$d$-dimensional face (having cardinality $d+1$) is called a $d$-face.

The \emph{join} of two simplicial complexes $\Delta_{1}$ and $\Delta_{2}$
on disjoint vertex sets $V_{1}$ and $V_{2}$ is a complex $\Delta_{1}*\Delta_{2}$
on vertex set $V_{1}\cup V_{2}$ with faces $\{\sigma_{1}\cup\sigma_{2}\,:\,\sigma_{i}\mbox{ a face of }\Delta_{i}\}$.

A \emph{clutter} $\mathcal{C}$ on a vertex set $V$ is a set of subsets
of $V$ (called \emph{circuits }or \emph{edges} of $\mathcal{C}$)
such that if $e_{1}$ and $e_{2}$ are distinct circuits of $\mathcal{C}$
then $e_{1}\not\subset e_{2}$. Clutters have also been referred to
in the literature as \emph{Sperner families}, or as \emph{antichains
of sets}. To avoid confusion with $1$-faces of a simplicial complex,
we will usually prefer the term ``circuit'' over ``edge''. A $d$-circuit
is a circuit consisting of exactly $d$ vertices, and a clutter is
\emph{$d$-uniform} if every circuit has exactly $d$ vertices. The
\emph{maximum circuit cardinality} of $\mathcal{C}$ is the largest
cardinality of any circuit in $\mathcal{C}$, and similarly for the
\emph{minimum circuit cardinality}.

An \emph{independent set }of $\mathcal{C}$ is a subset of $V$ containing
no circuit.\emph{ }Clutters and simplicial complexes are linked via
the independence complex $I(\mathcal{C})=\{\sigma\subseteq V\,:\,\sigma\mbox{ is an independent }\linebreak\mbox{set of }\mathcal{C}\}$,
and via the clutter of minimal non-faces $\mathcal{C}(\Delta)$. As
previously mentioned we have $\mathcal{C}(I(\mathcal{C}))=\mathcal{C}$
and $I(\mathcal{C}(\Delta))=\Delta$. 

There are two degenerate simplicial complexes on $V$: the simplicial
complex $\{\}$ with no faces, and the simplicial complex $\{\emptyset\}$
with only the empty set as a face. Notice that $\mathcal{C}(\{\})=\{\emptyset\}$,
while $\mathcal{C}(\{\emptyset\})=\{\{v\}\,:\, v\in V\}$.

Nondegenerate simplicial complexes admit a geometric realization,
a geometric simplicial complex with the same face incidences. When
we use geometric or topological terms such as dimension or homotopy
type to discuss simplicial complexes, we are referring to the geometric
realization.

All clutters and simplicial complexes considered are finite.

\subsection{Deletions and contractions, links}

Given a simplicial complex $\Delta$, two kinds of subcomplex are
of particular interest. If $\sigma$ is a face of $\Delta$, then
$\Delta\setminus\sigma$ is obtained from $\Delta$ by removing all
faces that contain $\sigma$ from the set system. The \emph{star }of
$\sigma$ is $\st_{\Delta}\sigma=\{\mbox{faces containing }\sigma\}$,
and the \emph{link} of $\sigma$ is the simplicial complex on vertex
set $V(\Delta)\setminus\sigma$ with faces 
\[
\link_{\Delta}\sigma=\{\tau\,:\,\sigma\cap\tau=\emptyset,\sigma\cup\tau\mbox{ is a face of }\Delta\},
\]

Thus $\link_{\Delta}\sigma$ is $\st_{\Delta}\sigma$ with all vertices
of $\sigma$ deleted, while $\st_{\Delta}\sigma=(\link_{\Delta}\sigma)*\sigma$.

Given a clutter $\mathcal{C}$, there are two ways of removing a vertex
that are of interest. Let $v\in V(\mathcal{C})$. The \emph{deletion}
$\mathcal{C}\setminus v$ is the clutter on the vertex set $V(\mathcal{C})\setminus\{v\}$
with circuits $\{e\,:\, e\mbox{ a circuit of }\mathcal{C}\mbox{ with }v\notin e\}$.
The \emph{contraction} $\mathcal{C}/v$ is the clutter on the vertex
set $V(\mathcal{C})\setminus\{v\}$ with circuits the minimal sets
of $\{e\setminus\{v\}\,:\, e\mbox{ a circuit of }\mathcal{C}\}$.
Thus, $\mathcal{C}\setminus v$ deletes all circuits containing $v$,
while $\mathcal{C}/v$ removes $v$ from each circuit containing it
(and then removes any redundant circuits). 

A clutter $\mathcal{D}$ obtained from $\mathcal{C}$ by repeated
deletion and/or contraction is called a \emph{minor} of $\mathcal{C}$.
If $\mathcal{D}$ is obtained only by deletions we call it an \emph{induced
subclutter} on vertex set $V(\mathcal{D})$. If $\mathcal{D}$ is
obtained only by contractions we call it a \emph{contraction} of $\mathcal{C}$.
Notice that if all the vertices contained in a circuit are contracted,
then what remains is the clutter $\{\emptyset\}$. It is straightforward
to prove that if $v\neq w$ are vertices then $(\mathcal{C}\setminus v)\setminus w=(\mathcal{C}\setminus w)\setminus v$,
$(\mathcal{C}/v)/w=(\mathcal{C}/w)/v$, and $(\mathcal{C}\setminus v)/w=(\mathcal{C}/w)\setminus v$.
\begin{example}
\label{exa:MatroidClutter} A clutter $\mathcal{C}$ is a \emph{matroid
circuit clutter} if it satisfies the \emph{weak circuit exchange property}:
that if $e_{1}$ and $e_{2}$ are circuits of $\mathcal{C}$ containing
a common vertex $v$, then there is some circuit $e_{3}$ such that
$e_{3}\subseteq(e_{1}\cup e_{2})\setminus\{v\}$. In this case the
deletion and contraction operations $\mathcal{C}\setminus v$ and
$\mathcal{C}/v$ are the usual deletion and contraction in a matroid
without loops or coloops. See \cite{Oxley:1992} for additional background
and other definitions of matroids. Duval \cite{Duval:2005} also has
a discussion of the relationship between minors in matroids and clutters/simplicial
complexes.
\end{example}
We collect the relationships between simplicial complex operations
and clutter operations in the following lemma.
\begin{lem}
Let $\mathcal{C},\mathcal{D}$ be clutters, and let $v$ be a vertex
of $V(\mathcal{C})$. Then 
\begin{enumerate}
\item $I(\mathcal{C}/v)=\link_{I(\mathcal{C})}v$.
\item $I(\mathcal{C}\setminus v)=I(\mathcal{C})\setminus v$, considered
as a simplicial complex on $V(\mathcal{C})\setminus v$.
\item If $\mathcal{C}'$ consists of the minimal sets of $\mathcal{C}\cup\{\sigma\}$
(where $\sigma$ is an independent set), then $I(\mathcal{C}')=I(\mathcal{C})\setminus\sigma$.
\item $I(\mathcal{C}\disjointunion\mathcal{D})=I(\mathcal{C})*I(\mathcal{D})$.
\end{enumerate}
\end{lem}
There are obvious dual versions of these which relate a simplicial
complex $\Delta$ with its minimal non-face clutter $\mathcal{C}(\Delta)$.
\begin{example}
If $G$ is a graph (i.e., a clutter with every circuit having cardinality
2), then let $N[v]=\{v\mbox{ and all its neighbors}\}$. Then $G\setminus v$
is the usual induced subgraph, while $G/v$ is the induced subgraph
$G\setminus N[v]$ together with a singleton circuit for each neighbor
$w$ to $v$. In particular, $I(G/v)=I(G\setminus N[v])$. \end{example}
\begin{rem}
Contraction in a graph $G$ as a clutter should not be confused with
contraction in the circuit matroid of $G$! 
\end{rem}

\begin{rem}
In a graph, every contraction operation can be expressed (up to singleton
circuits) as repeated deletion operations, deleting each vertex in
$N[v]$. This does not hold true in general clutters. For example,
if $\mathcal{C}$ consists of a single circuit of cardinality 3, then
deleting any vertex leaves a clutter with no circuits, while contracting
any vertex leaves a circuit of cardinality 2.
\end{rem}

\subsection{Shellable and Cohen-Macaulay complexes}

We recommend the unfamiliar reader to \cite{Bjorner/Wachs:1996,Bjorner/Wachs:1997}
for background on shellability, and to \cite{Bjorner/Wachs/Welker:2009,Bruns/Herzog:1993,Stanley:1996}
for Cohen-Macaulay and sequentially Cohen-Macaulay simplicial complexes,
but give a brief review here.

\medskip{}

Let $k$ be any field or the ring of integers. A simplicial complex
is \emph{Cohen-Macaulay} if, for every face $\sigma$ (including $\sigma=\emptyset$),
we have $\tilde{H}_{i}(\link_{\Delta}\sigma,k)=0$ for $i<\dim(\link_{\Delta}\sigma)$.
An equivalent condition is for the Stanley-Reisner ring of $\Delta$
to be a Cohen-Macaulay ring, and there is also an equivalent topological
condition which makes no reference to the face structure of $\Delta$.
Examples of complexes that are Cohen-Macaulay (over any $k$) include
any triangulation of a sphere.

Every Cohen-Macaulay simplicial complex $\Delta$ is \emph{pure},
i.e., every pair of facets have the same dimension. There is a related
notion for general complexes. The \emph{pure $d$-skeleton} of a simplicial
complex $\Delta$ is the subcomplex generated by all faces of dimension
$d$. We say that a complex is \emph{sequentially Cohen-Macaulay over
$k$} if the pure $d$-skeleton is Cohen-Macaulay (over $k$) for
all $d$. Once again, there are equivalent ring-theoretic and topological
conditions for the sequentially Cohen-Macaulay property. Any pure
sequentially Cohen-Macaulay complex is Cohen-Macaulay.

A simplicial complex $\Delta$ is \emph{shellable} if there is an
ordering $\sigma_{1},\dots,\sigma_{m}$ of the facets of $\Delta$
such that the intersection of $\sigma_{i}$ with the complex generated
by $\sigma_{1},\dots,\sigma_{i-1}$ is pure $(\dim\sigma_{i}-1)$-dimensional.
When possible, we avoid this definition and work through the condition
of $k$-decomposability introduced in Section \ref{sec:k-decomposableComplexes}. 

Any shellable complex is sequentially Cohen-Macaulay over any $k$,
and we view shellability as a combinatorial condition for a complex
to be sequentially Cohen-Macaulay. Since the results we prove will
be independent of the field or ring $k$, we henceforth suppress $k$
from our notation.

\medskip{}

A \emph{linear resolution} of an ideal $I$ in a ring $R$ is a minimal
free resolution of $R/I$ satisfying certain properties --- the exact
definition will not be important to us, as we work through the characterization
of Eagon and Reiner \cite{Eagon/Reiner:1998} that the property of
possessing a linear resolution is Alexander-dual to being Cohen-Macaulay.
As the details are somewhat complicated, and only required in Section
\ref{sec:LinearResolutions}, we defer further discussion to that
section.

\section{\label{sec:k-decomposableComplexes}$k$-decomposable complexes}

Provan and Billera \cite{Provan/Billera:1980} introduced a definition
of $k$-decomposability for pure complexes. For $k=0$ these are known
as vertex decomposable complexes, and the definition of vertex decomposable
complexes was extended to non-pure complexes by Björner and Wachs
\cite{Bjorner/Wachs:1997}. We now give an analogous extension of
$k$-decomposability to non-pure complexes for $k>0$.

The following definition was first made by Jonsson \cite[Definition 2.10]{Jonsson:2005}:
\begin{defn}
\label{def:SheddingFace} Let $\Delta$ be a simplicial complex on
vertex set $V$. Then a face $\sigma$ is called a \emph{shedding
face} if every face $\tau$ of $\st_{\Delta}\sigma$ satisfies the
following exchange property: for every $v\in\sigma$ there is a $w\in V\setminus\tau$
such that $(\tau\cup\{w\})\setminus\{v\}$ is a face of $\Delta$.\end{defn}
\begin{rem}
An equivalent condition to the exchange property of Definition \ref{def:SheddingFace}
is the following: no facet of $(\st_{\Delta}\sigma)\setminus\sigma$
is a facet of $\Delta\setminus\sigma$.
\end{rem}

\begin{rem}
In the case where $\sigma$ is a single vertex, the definition of
shedding vertex specializes to that of Björner and Wachs \cite[Section 11]{Bjorner/Wachs:1997}.
In the case that $\Delta$ is pure, the definition specializes to
that of Provan and Billera \cite[Definition 2.1]{Provan/Billera:1980}.
\end{rem}
Our main fact about shedding faces is the following generalization
of \cite[Lemma 6]{Wachs:1999b}:
\begin{lem}
\label{lem:SheddingFaceShelling} \emph{(Essentially Jonsson \cite{Jonsson:2005})}
If $\sigma$ is a shedding face for a simplicial complex $\Delta$
such that both $\Delta\setminus\sigma$ and $\link_{\Delta}\sigma$
are shellable, then $\Delta$ is shellable.\end{lem}
\begin{proof}[Sketch.]
 We first order the facets not containing $\sigma$ according to
the shelling order of $\Delta\setminus\sigma$, followed by the facets
containing $\sigma$ in the order indicated by the shelling of $\link_{\Delta}\sigma$.
It is straightforward to verify this facet ordering is a shelling.
\end{proof}
We see that the definition of shedding face can be viewed as a tool
to build up a shelling of $\Delta$ by ``sorting'' the facets of
$\Delta$. 
\begin{defn}
\label{def:kDecomposable} A simplicial complex $\Delta$ is recursively
defined to be $k$-decomposable if either $\Delta$ is a simplex or
else has a shedding face $\sigma$ with $\dim\sigma\leq k$ such that
both $\Delta\setminus\sigma$ and $\link_{\Delta}\sigma$ are $k$-decomposable.
We consider the degenerate complexes $\{\}$ and $\{\emptyset\}$
to be $k$-decomposable for all $k\geq-1$.
\end{defn}
Definition \ref{def:kDecomposable} obviously extends the definition
of vertex decomposability and pure $k$-decomposability. 

Many of the theorems proved by Provan and Billera go through straightforwardly
for our definition. Most interesting from the author's perspective
is:
\begin{thm}
\label{thm:dDimlShellableIsdDecomp} \emph{(Jonsson }\cite{Jonsson:2005}\emph{)}
A $d$-dimensional (not necessarily pure) simplicial complex $\Delta$
is shellable if and only if it is $d$-decomposable.\end{thm}
\begin{proof}
Lemma \ref{lem:SheddingFaceShelling} gives the $(\Leftarrow)$ direction.

Conversely, it follows directly from \cite[Lemma 2.4]{Bjorner/Wachs:1996}
that the ``minimal new face'' contained in the last face of a shelling
order is a shedding face. Induction then gives the $(\Rightarrow)$
direction.
\end{proof}
Theorem \ref{thm:dDimlShellableIsdDecomp} tells us that $k$-decomposability
gives a hierarchical structure on the family of shellable complexes:
every $k$-decomposable complex is also $(k+1)$-decomposable, and
every shellable $d$-dimensional complex is $j$-decomposable for
some $j\leq d$.

Some other theorems that extend easily are:
\begin{prop}
\label{pro:LinkKdecomp} If a simplicial complex $\Delta$ is $k$-decomposable,
then for every face $\tau$ of $\Delta$ it holds that $\link_{\Delta}\tau$
is $k$-decomposable.\end{prop}
\begin{proof}
Entirely similar to \cite[Proposition 2.3]{Provan/Billera:1980}.
\end{proof}
The following is stronger than \cite[Proposition 2.4]{Provan/Billera:1980}:
\begin{prop}
\label{pro:JoinOfKdecIsKdec} The simplicial complexes $\Delta_{1}$
and $\Delta_{2}$ are $k$-decomposable if and only if $\Delta_{1}*\Delta_{2}$
is $k$-decomposable.\end{prop}
\begin{proof}
The ``only if'' direction is entirely similar to \cite[Proposition 2.4]{Provan/Billera:1980}.

Conversely, suppose that $\sigma$ is a shedding face of $\Delta_{1}*\Delta_{2}$,
with $\sigma_{i}=V(\Delta_{i})\cap\sigma$. Let $\tau_{1}$ be any
facet of $\Delta_{1}$ which contains $\sigma_{1}$, and $\tau_{2}$
similarly for $\Delta_{2}$ and $\sigma_{2}$. Then every face of
the form $\rho_{1}\disjointunion\tau_{2}$ with $\sigma_{1}\subset\rho_{1}$
satisfies the shedding face exchange property, and (since $\tau_{2}$
is a facet) every face $\rho_{1}$ of $\Delta_{1}$ containing $\sigma_{1}$
satisfies the exchange property. Thus $\sigma_{1}$ and by symmetry
$\sigma_{2}$ are shedding faces for $\Delta_{1}$ and $\Delta_{2}$.
Since $\sigma_{1}$ and $\sigma_{2}$ are contained in $\sigma$,
both have dimension $\leq k$, and at least one is non-empty.

Next, we notice that $\link_{\Delta}\sigma=\link_{\Delta_{1}}\sigma_{1}*\link_{\Delta_{2}}\sigma_{2}$,
and by induction each of $\link_{\Delta_{1}}\sigma_{1}$ and $\link_{\Delta_{2}}\sigma_{2}$
is $k$-decomposable. (If $\sigma_{i}=\emptyset$, then we notice
that $\link_{\Delta_{i}}\sigma_{i}=\Delta_{i}$.) Finally, $\Delta_{1}\setminus\sigma_{1}=\link_{\Delta\setminus\sigma}\tau_{2}$,
which is $k$-decomposable by Proposition \ref{pro:LinkKdecomp}.
\end{proof}
Our main application of $k$-decomposability will come in Section
\ref{sec:Shellability}, where we use it to prove that the independence
complex of a chordal clutter is shellable.
\begin{rem}
Simon \cite[Section 2.3]{Simon:1994} has introduced ``clean ideal
trees,'' an extension of $k$-decomposability via commutative algebra;
however the concrete condition for a shedding face seems to better
lend itself to constructing shellings.
\end{rem}
We prove the following lemma for use in Section \ref{sec:LinearResolutions}. 
\begin{lem}
\label{lem:VertDecSkeleton}Let $\Delta$ be a vertex decomposable
simplicial complex. Then the $s$-skeleton of $\Delta$ is vertex
decomposable for any $s$.\end{lem}
\begin{proof}
Let $\Delta^{(s)}$ denote the $s$-skeleton of a simplicial complex.
Clearly, $\link_{\Delta^{(s)}}v=(\link_{\Delta}v)^{(s-1)}$, while
$\Delta^{(s)}\setminus v=(\Delta\setminus v)^{(s)}$, so that by induction
it suffices to produce a shedding vertex in $\Delta^{(s)}$. Then
either $\Delta$ is a simplex, in which case every vertex of $\Delta^{(s)}$
is a shedding vertex; or else $\Delta$ has a shedding vertex, which
is easily seen to remain a shedding vertex in $\Delta^{(s)}$.
\end{proof}
We close this section with a question, which we believe to be open
even in the case of flag complexes ($m=2$).
\begin{question}
What is the smallest $k$ such that if $\Delta$ is a shellable $d$-dimensional
complex with $\mathcal{C}(\Delta)$ having maximum circuit cardinality
$m$, then $\Delta$ is necessarily $k$-decomposable?
\end{question}
It is particularly natural to ask if every such complex is $(m-1)$-decomposable,
as deleting an $(m-1)$-face preserves the maximum circuit cardinality
condition.

\section{\label{sec:ChordalClutters}Chordal clutters}

Before introducing our definition of chordal clutters, we recall the
definition and main structure theorem for chordal graphs. A graph
is \emph{chordal} if every induced cyclic subgraph of $G$ has length
$3$. A vertex $v$ of $G$ is \emph{simplicial} if the neighborhood
of $v$ in $G$ is a complete subgraph. The main theorem characterizing
chordal graphs is:
\begin{thm}
\emph{(essentially G.~Dirac \cite{Dirac:1961})} A graph $G$ is
chordal if and only if every induced subgraph of $G$ has a simplicial
vertex.
\end{thm}
Most of the attempts in algebraic combinatorics at extending the definition
of chordal to clutters have centered around extending the definition
of simplicial vertex, and ours will be no exception.
\begin{defn}
\label{def:SimplicialVertexClutter} Let $\mathcal{C}$ be a clutter.
A vertex $v$ of $\mathcal{C}$ is \emph{simplicial} if for every
two circuits $e_{1}$ and $e_{2}$ of $\mathcal{C}$ that contain
$v$, there is a third circuit $e_{3}$ such that $e_{3}\subseteq(e_{1}\cup e_{2})\setminus\{v\}$.
\end{defn}
In the case where $G$ is a graph, Definition \ref{def:SimplicialVertexClutter}
obviously agrees with the previous definition of a simplicial vertex.
\begin{defn}
A clutter $\mathcal{C}$ is \emph{chordal} if every minor of $\mathcal{C}$
has a simplicial vertex. \end{defn}
\begin{example}
The following clutters are chordal:
\begin{enumerate}
\item Chordal graphs: If $G$ is a graph, then $G/v$ is (up to singleton
circuits) the induced subgraph $G\setminus N[v]$. Hence the definition
of chordal clutter specializes in graphs to the usual definition of
chordal.
\item The \emph{complete $d$-uniform clutter} $\mathcal{K}_{n}^{d}$ is
the clutter with $n$ vertices and circuit set ${V \choose d}$. Since
$\mathcal{K}_{n}^{d}/v\cong\mathcal{K}_{n-1}^{d-1}$, $\mathcal{K}_{n}^{d}\setminus v\cong\mathcal{K}_{n-1}^{d}$,
and every vertex is simplicial, the complete $d$-uniform clutter
is chordal.
\item Matroid circuits: Compare Definition \ref{def:SimplicialVertexClutter}
with Example \ref{exa:MatroidClutter}. The simplicial vertex condition
is exactly the weak circuit exchange property of matroids at a single
vertex $v$. Thus every vertex of a matroid circuit clutter is simplicial.
Since every deletion or contraction of a matroid gives another matroid,
the circuit clutter of any matroid is chordal.\\
For example, $\mathcal{K}_{n}^{d}$ is the circuit clutter of a uniform
matroid.
\end{enumerate}
\end{example}

\begin{example}
\label{exa:FreeVertexProperty} Van Tuyl and Villareal \cite{VanTuyl/Villarreal:2008}
define a clutter $\mathcal{C}$ to have the \emph{free vertex property}
if every minor of $\mathcal{C}$ has a \emph{free vertex}, that is,
a vertex appearing in exactly one circuit of $\mathcal{C}$. We observe
that a free vertex is simplicial, so clutters with the free vertex
property are chordal. Clutters with the free vertex property were
shown to be shellable in \cite[Theorem 5.3]{VanTuyl/Villarreal:2008},
and a restricted case of Proposition \ref{pro:SimplicialAndContractionsGivesShellable}
for free vertices was shown in \cite[Theorem 2.8]{Morey/Reyes/Villareal:2008}.

Van Tuyl and Villarreal notice \cite[Corollary 5.7]{VanTuyl/Villarreal:2008}
that if $G$ is a chordal graph, then the clutter $\mathcal{C}$ with
vertices $V(G)$ and circuits consisting of all cliques in $G$ has
the free vertex property. The Graham-Yu-Özsoyglu algorithm from database
theory can be used to show that every clutter with the free vertex
property has this form: a helpful reference is \cite[especially Theorem 3.4]{Beeri/Fagin/Maier/Yannakakis:1983}.
Specifically, the Graham-Yu-Özsoyglu algorithm chooses a free vertex
$v$ contained in a unique circuit $e$, and deletes $v$ if $e\setminus\{v\}$
is strictly contained in another circuit, and contracts $v$ otherwise
--- the algorithm terminates if and only if $\mathcal{C}$ is the
clutter of cliques of a chordal graph.\end{example}
\begin{rem}
Clutters (and more generally hypergraphs) which have the free vertex
property have often been referred to as ``acyclic''. Since despite
the name these clutters may have cycles, we prefer the free vertex
property terminology.\end{rem}
\begin{example}
The clutter with circuits $\{1,2,3\},\{1,4,5\},\{2,3,4,5\},\{2,3,6\},\{4,5,6\}$
has simplicial vertex $1$, and is easily verified to be chordal;
but is not a chordal graph or matroid circuit clutter, and does not
have the free vertex property.
\end{example}

\begin{example}
\label{exa:EmtanderChordal} Emtander \cite{Emtander:2010}, extending
ideas from Hà and Van Tuyl \cite{Ha/VanTuyl:2008}, has a different
but related definition of chordal for $d$-uniform clutters. Let a
vertex $v$ be a \emph{complete-neighborhood vertex} if the induced
subclutter on $S=\{x\,:\, x,v\in e\}$ is the complete $d$-uniform
clutter, i.e. has circuits ${S \choose d}$. Emtander calls a $d$-uniform
clutter ``chordal'' if every induced subclutter either has a complete-neighborhood
vertex, or else no circuits. 

A complete-neighborhood vertex is clearly simplicial in our sense,
but Emtander requires only deletions to have simplicial vertices,
while we require both deletions and contractions. Examples which are
chordal in our sense but not in Emtander's are easy to come by (most
matroids will do). An example which has complete-neighborhood vertices
in every induced subclutter but is not chordal is the clutter $\mathcal{C}$
with circuits $\{1,2,3\},\{3,4,5\},\{5,6,7\},\{7,8,1\}$. Every induced
subclutter of $\mathcal{C}$ has the free vertex property, but contracting
$2$, $4,$ $6$, and $8$ leaves the cyclic graph $C_{4}$. (It follows
immediately that $I(\mathcal{C})$ is not shellable or sequentially
Cohen-Macaulay.) 
\end{example}

\section{\label{sec:Shellability}Shellability of the independence complex}

Our main goal of this section will be to prove Theorem \ref{thm:ChordalClutterShellable}.

Recall \cite[Lemma 6]{Woodroofe:2009a} that if $G$ is a graph with
vertices $v$ and $w$ such that $N[v]\subseteq N[w]$, then $w$
is a shedding vertex in $I(G)$. Motivated by this result, we define
a \emph{neighborhood containment pair }of a clutter $\mathcal{C}$
to be a vertex $v$ and a circuit $e$ with $v\in e$ such that if
$v\in e_{2}$ for any circuit $e_{2}\neq e$, then there exists an
circuit $e_{3}\subseteq(e\cup e_{2})\setminus v$. Thus, a simplicial
vertex forms a neighborhood containment pair with any circuit containing
it.
\begin{lem}
\label{lem:NbhdContainment}If $\mathcal{C}$ is a clutter with a
neighborhood containment pair $(v,e)$ then $\sigma=e\setminus\{v\}$
is a shedding face of $I(\mathcal{C})$.\end{lem}
\begin{proof}
Suppose that $\tau$ is a face of $I(\mathcal{C})$ containing $\sigma$.
Then $\tau\cup v$ contains $e$, so is not a face (and in particular
$v\notin\tau$). For $x\in\sigma$, if $(\tau\setminus x)\cup v$
is not a face of $I(\mathcal{C})$ (hence of $I(\mathcal{C})\setminus\sigma$),
then $(\tau\setminus x)\cup v$ contains some circuit  $e_{2}$ with
$v\in e_{2}$. But then the neighborhood containment condition gives
an $e_{3}\subseteq(e\cup e_{2})\setminus v\subseteq\tau$, contradicting
the choice of $\tau$ as a face. Hence any such $x$ can be exchanged
for $v$, fulfilling the shedding face exchange axiom.\end{proof}
\begin{prop}
\label{pro:SimplicialAndContractionsGivesShellable}If $\mathcal{C}$
is a clutter containing a simplicial vertex $v$, and if every proper
contraction of $\mathcal{C}$ is shellable, then $\mathcal{C}$ is
shellable.\end{prop}
\begin{proof}
Let $e_{1},\dots,e_{k}$ be the circuits containing $v$, and $\sigma_{1},\dots,\sigma_{k}$
be the associated shedding faces $e_{i}\setminus\{v\}$. Let $\mathcal{C}_{0}=\mathcal{C}$,
and $\mathcal{C}_{i}$ be generated by the minimal sets of $\mathcal{C}\cup\{\sigma_{1},\dots,\sigma_{i}\}$,
so that $I(\mathcal{C}_{i})=I(\mathcal{C})\setminus\sigma_{1}\setminus\dots\setminus\sigma_{i}$. 

Then since there is an $e'\subseteq(e_{i}\cup e_{j})\setminus v=\sigma_{i}\cup\sigma_{j}$,
we get that $\mathcal{C}_{i-1}$ has some $e''\subseteq e'\subseteq\sigma_{i}\cup\sigma_{j}$.
In $\mathcal{C}_{i-1}/\sigma_{i}$ the circuit $e''$ contracts to
$e'''\subseteq\sigma_{j}$. In particular the minimal sets of $\mathcal{C}_{i-1}/\sigma_{i}$
are the same as those of $\mathcal{C}/\sigma_{i}$. We have shown
that $\mathcal{C}_{i-1}/\sigma_{i}=\mathcal{C}/\sigma_{i}$.

It is straightforward to check that $v$ is simplicial in each of
$\mathcal{C}_{1},\dots,\mathcal{C}_{k-1}$, and that the circuits
of $\mathcal{C}_{i}$ containing $v$ are $e_{i+1},\dots,e_{k}$.
By Lemma \ref{lem:NbhdContainment} we have that $\sigma_{i}$ is
a shedding face in $I(\mathcal{C}_{i-1})$. Every required link of
the form $\link_{I(\mathcal{C}_{i-1})}\sigma_{i}=I(\mathcal{C}_{i-1}/\sigma_{i})=I(\mathcal{C}/\sigma_{i})$
is shellable. The vertex $v$ is isolated in $I(\mathcal{C}_{k})$,
so that $I(\mathcal{C}_{k})=I(\mathcal{C}/v)*v$ is shellable; while
each $I(\mathcal{C}_{i})$ is shellable by Lemma \ref{lem:SheddingFaceShelling}
and induction.
\end{proof}
A short intuitive explanation of the proof of Proposition \ref{pro:SimplicialAndContractionsGivesShellable}
is that the faces $\sigma_{i}=e_{i}\setminus\{v\}$ are exactly the
circuits added when $v$ is contracted, so that deleting all of the
$\sigma_{i}$'s from $I(\mathcal{C})$ leaves $\left(\mathcal{C}/v\right)\disjointunion\{v\}$
as the minimal non-face clutter. 
\begin{cor}
\label{cor:ContractionSimplicial}Let $\mathcal{C}$ be a clutter
with maximum circuit cardinality $k$, such that every contraction
of $\mathcal{C}$ has a simplicial vertex. Then $I(\mathcal{C})$
is $(k-2)$-decomposable, hence shellable and sequentially Cohen-Macaulay.\end{cor}
\begin{proof}
By induction and noting that each shedding face produced in Proposition
\ref{pro:SimplicialAndContractionsGivesShellable} has dimension at
most $k-2$.
\end{proof}
We thus have the following specialization of Theorem \ref{thm:ChordalClutterShellable}:
\begin{cor}
If $\mathcal{C}$ is a chordal clutter with maximum circuit cardinality
$k$, then $I(\mathcal{C})$ is $(k-2)$-decomposable, hence shellable
and sequentially Cohen-Macaulay.
\end{cor}
We also break out the statement of Corollary \ref{cor:ContractionSimplicial}
in the case where $\mathcal{C}$ is a graph. 
\begin{cor}
\label{cor:NbhdSimplGraph}If $G$ is a graph such that $G\setminus N[A]$
has a simplicial vertex for any independent set $A$, then $G$ is
vertex decomposable.
\end{cor}
The family of graphs given in Corollary \ref{cor:NbhdSimplGraph}
is a considerably more general family than that of chordal graphs,
including for example simplicial graphs \cite{Cheston/Hare/Hedetniemi/Laskar:1988},
the family of graphs considered in \cite[Theorem 3.2]{Francisco/Ha:2008},
etc.

\section{\label{sec:LinearResolutions}Linear resolutions and Alexander duality}

Recall that the complement $\overline{G}$ of a graph $G$ is the
graph with the same vertex set and with circuit set $\{xy\,:\, xy\not\in E(G)\}$.
We define the complement of a $d$-uniform clutter similarly for any
$d$. An important theorem in the algebraic combinatorics of a chordal
graph is:
\begin{thm}
\emph{(Fröberg \cite{Froberg:1990}) \label{thm:FrobergLinRes}} Let
$G$ be a graph. Then the circuit ideal of $G$ has a linear resolution
if and only if $\overline{G}$ is chordal.
\end{thm}
In this section we generalize the ``if'' direction of Theorem \ref{thm:FrobergLinRes}
from chordal graphs to chordal clutters, in particular proving Theorem
\ref{thm:ClutterLinearRes}. We will first need to recall some facts
about Alexander duality.

\subsection{Review of Alexander duality}

The \emph{Alexander dual} of a simplicial complex $\Delta$ (denoted
$\alexdual{\Delta}$) is the simplicial complex with vertices $V=V(\Delta)$
and facets $\{V\setminus e\,:\, e\mbox{ a circuit of }\mathcal{C}(\Delta)\}$.
The vertex set that we consider $\Delta$ over has unusually great
importance in this definition, and if we wish to emphasize the vertex
set that we are operating over we will use a subscript, e.g. $\alexdual{\Delta_{V}}$.

The Alexander dual allows us to reduce the question of the existence
of a linear resolution to topological combinatorics:
\begin{thm}
\emph{\label{thm:Eagon-ReinerLinearRes} (Eagon and Reiner \cite[Theorem 3]{Eagon/Reiner:1998})
}Let $\Delta$ be a simplicial complex. The circuit ideal of $\Delta$
has a linear resolution if and only if $\alexdual{\Delta}$ is Cohen-Macaulay.
\end{thm}
In particular, one approach to proving the ``if'' direction of Theorem
\ref{thm:FrobergLinRes} is as follows:
\begin{thm}
\emph{\label{thm:Eagon-ReinerVD} (Eagon and Reiner \cite[Proposition 8]{Eagon/Reiner:1998})}
If $G$ is a chordal graph, then $\alexdual{I(\overline{G})}$ is
vertex decomposable.
\end{thm}
If $\alexdual{\Delta}$ is shellable, then the circuit ideal is said
to have \emph{linear quotients}.
\begin{rem}
Theorem \ref{thm:Eagon-ReinerLinearRes} tells us that classifying
the circuit ideals with linear resolution is equivalent to classifying
all Cohen-Macaulay complexes, which is likely intractable. Finding
large classes of circuit ideals with linear resolutions remains an
interesting problem.
\end{rem}
We recall some standard facts about Alexander duality \cite{Cornuejols:2001,Miller/Sturmfels:2005}:
\begin{lem}
\label{lem:AlexDualitySummary}If $\Delta$ is any simplicial complex
on vertex set $V$ then
\begin{enumerate}
\item $\tilde{H}_{i}(\Delta)\cong\tilde{H}^{\vert V\vert-i-3}(\alexdual{\Delta})$.
\item $\alexdual{(\alexdual{\Delta})}=\Delta$. 
\item $\alexdual{(\Delta\setminus v)}=\link_{\alexdual{\Delta}}v$ and $\alexdual{(\link_{\Delta}v)}=\alexdual{\Delta}\setminus v$.
\item $\Delta$ is pure of dimension $d$ if and only if $\mathcal{C}(\alexdual{\Delta})$
is $(\vert V\vert-d-1)$-uniform.
\end{enumerate}
\end{lem}
\begin{rem}
The Alexander dual has been studied in topological combinatorics at
least as far back as \cite[Section 6]{Kalai:1983}. It has also been
studied in the context of combinatorial optimization under the name
\emph{blocker} or \emph{transversal}, and it is in this context that
Lemma \ref{lem:AlexDualitySummary} parts (2) and (3) were first observed.
We refer the reader to \cite{Cornuejols:2001} for further background
and references from the combinatorial optimization point of view,
or to \cite{Miller/Sturmfels:2005} from the algebraic combinatorics
point of view.
\end{rem}

\subsection{Alexander duals of complements to chordal clutters}

If $\mathcal{C}$ is a clutter, then define $c_{d}(\mathcal{C})$
to be the clutter with the same vertex set $V$ as $\mathcal{C}$
and circuit set $\{e\subseteq V\,:\,\vert e\vert=d,\, e\mbox{ not a circuit of }\mathcal{C}\}.$
In the special case that $\mathcal{C}$ is $d$-uniform, this is the
complement\emph{ }of $\mathcal{C}$. We refer to the circuits of $c_{d}(\mathcal{C})$
as $d$-non-circuits of $\mathcal{C}$. 

We start by relating contraction in $\mathcal{C}$ with contraction
in $c_{d}(\mathcal{C})$:
\begin{lem}
\label{lem:dComplementContraction} Let $\mathcal{C}$ be a clutter
with no circuits of cardinality $(d-1)$, and $v$ be a simplicial
vertex. Then $c_{d}(\mathcal{C})/v=c_{d-1}(\mathcal{C}/v)$.\end{lem}
\begin{proof}
Suppose by contradiction that $e$ is a $d$-non-circuit of $\mathcal{C}$
with $v\notin e$, and that $e$ is the only such non-circuit contained
in the set $e\cup v$. Then the induced subclutter of $\mathcal{C}$
on the set $e\cup\{v\}$ is a complete clutter with one circuit removed
($\mathcal{K}_{d+1}^{d}\setminus\{e\}$), which contradicts the hypothesis
that $v$ is simplicial. It follows that every $d$-non-circuit of
$\mathcal{C}$ contains a $(d-1)$-set $e'$ which is a circuit of
$c_{d}(\mathcal{C})/v$, i.e. such that $e'\cup\{v\}$ is a non-circuit
of $\mathcal{C}$. Thus such $e'$ are precisely the circuits of $c_{d}(\mathcal{C})/v$.

Because there are no circuits with $d-1$ vertices in $\mathcal{C}$,
the $(d-1)$-circuits of $\mathcal{C}/v$ are exactly the sets $e$
with $e\cup\{v\}$ a circuit of $\mathcal{C}$. We have that 
\[
\{e\,:\,\vert e\vert=d-1,\, e\cup\{v\}\mbox{ a non-circuit of }\mathcal{C}\}=c_{d-1}(\mathcal{C}/v)=c_{d}(\mathcal{C})/v.\qedhere
\]

\end{proof}
Notice that $\mathcal{C}/v$ is in general not a uniform clutter,
even if the starting clutter $\mathcal{C}$ was uniform. It is for
this reason that we work with $c_{d}$, which is defined for every
clutter, rather than with a more straightforward complement of $d$-uniform
clutters.
\begin{lem}
\label{lem:dComplementShedding} If $v$ is a simplicial vertex of
a clutter $\mathcal{C}$ such that $\mathcal{C}\setminus v$ has at
least one $d$-non-circuit ($d\geq2$), then $v$ is a shedding vertex
in $\alexdual{I(c_{d}(\mathcal{C}))}$.\end{lem}
\begin{proof}
Suppose that $\sigma$ is a facet of $\link_{\alexdual{I(c_{d}(\mathcal{C}))}}v$,
so that $\sigma=\left(V\setminus e\right)\setminus\{v\}$ for some
$d$-non-circuit $e$ not containing $v$. (Such a facet exists by
the condition requiring $\mathcal{C}\setminus v$ to have at least
one $d$-non-circuit.) Since $d\geq2$ there are vertices $w_{1},w_{2}\in e$,
and we let $e_{i}$ be the set $(e\setminus w_{i})\cup\{v\}$. 

If both $e_{1}$ and $e_{2}$ are circuits of $\mathcal{C}$, then
$(e_{1}\cup e_{2})\setminus v=e$ is also a circuit, a contradiction;
so at least one $e_{i}$ is a $d$-non-circuit. But then $\tau=V\setminus e_{i}$
is a facet of $\alexdual{I(c_{d}(\mathcal{C}))}\setminus\{v\}$ with
$\tau=\sigma\cup\{w_{i}\}$, meeting the requirement for a shedding
vertex.
\end{proof}
We are now ready to prove:
\begin{thm}
\label{thm:dComplementChordal_VD}If $\mathcal{C}$ is a chordal clutter
with minimum circuit cardinality $d$, then $\alexdual{I(c_{d}(\mathcal{C}))}$
is vertex decomposable.\end{thm}
\begin{proof}
We proceed by induction, with base cases as follows: If $c_{d}(\mathcal{C})$
has no circuits, then $\alexdual{I(c_{d}(\mathcal{C}))}$ is the degenerate
complex $\{\}$, which we defined to be vertex decomposable. If $d=1$
and there is a circuit in $c_{d}(\mathcal{C})$, then the facets of
$\alexdual{I(c_{d}(\mathcal{C}))}$ are some collection of codimension
$1$ faces of a simplex, hence vertex decomposable \cite[proof of Proposition 8]{Eagon/Reiner:1998}.

For $d>1$, let $v$ be a simplicial vertex of $\mathcal{C}$. Then
\[
\link_{\alexdual{I(c_{d}(\mathcal{C}))}}v=\alexdual{I(c_{d}(\mathcal{C})\setminus v)}=\alexdual{I(c_{d}(\mathcal{C}\setminus v))}
\]
 is vertex decomposable by induction, and 
\[
\alexdual{I(c_{d}(\mathcal{C}))}\setminus v=\alexdual{I(c_{d}(\mathcal{C})/v)}=\alexdual{I(c_{d-1}(\mathcal{C}/v))}
\]
 is vertex decomposable by induction with Lemma \ref{lem:dComplementContraction}
and minimality of $d$.

If $\mathcal{C}\setminus v$ has a $d$-non-circuit, then $v$ is
a shedding vertex by Lemma \ref{lem:dComplementShedding}, hence $\alexdual{I(c_{d}(\mathcal{C}))}$
is vertex decomposable. Otherwise, $v$ is contained in every circuit
of $c_{d}(\mathcal{C})$, hence in no facet of $\alexdual{I(c_{d}(\mathcal{C}))}$,
so that $\alexdual{I(c_{d}(\mathcal{C}))}=\alexdual{I(c_{d}(\mathcal{C}))}\setminus v$,
which is vertex decomposable by induction.
\end{proof}
We have proved the following generalization of Theorem \ref{thm:ClutterLinearRes}.
\begin{cor}
\label{cor:dComplementChordal_LinRes}If $\mathcal{C}$ is a chordal
clutter with minimum circuit cardinality $d$, then the circuit ideal
of $c_{d}(\mathcal{C})$ has linear quotients, hence a linear resolution.
\end{cor}
As mentioned in Example \ref{exa:EmtanderChordal}, there are clutters
such that every subclutter contains a complete-neighborhood vertex,
but that are not chordal. We can however use a similar technique to
show that clutters with a complete-neighborhood vertex in every induced
subclutter are vertex decomposable, improving the previous result
\cite[Theorem 4.3]{Emtander/Mohammadi/Moradi:2011} that such clutters
are shellable:
\begin{prop}
Let $\mathcal{C}$ be a $d$-uniform clutter such that every induced
subclutter has a complete-neighborhood vertex. Then $\alexdual{I(c_{d}(\mathcal{C}))}$
is vertex decomposable.\end{prop}
\begin{proof}
By induction we may assume that $\alexdual{\link_{\alexdual{I(c_{d}(\mathcal{C}))}}v=I(c_{d}(\mathcal{C}\setminus v))}$
is shellable. A complete-neighborhood vertex $v$ is simplicial, thus
either $v$ is a shedding vertex or else $\alexdual{I(c_{d}(\mathcal{C}))}=\alexdual{I(c_{d}(\mathcal{C}))}\setminus v$,
exactly as in the proof of Theorem \ref{thm:dComplementChordal_VD}.
It remains only to show that $\alexdual{I(c_{d}(\mathcal{C})/v)}$
is shellable. 

Let $N=\bigcup_{v\in e}(e\setminus\{v\})$ be the neighborhood of
$v$. The induced subclutter on $N$ is $\mathcal{K}_{\vert N\vert}^{d}$.
By Lemma \ref{lem:dComplementContraction}, $\alexdual{I(c_{d}(\mathcal{C})/v)}$
has circuits $\{e\,:\,\vert e\vert=d-1,\, e\cup\{v\}\mbox{ a non-circuit of }\mathcal{C}\}$,
that is, all $e$ of cardinality $d-1$ such that $e\not\subseteq N$.

It follows that $\alexdual{I(c_{d}(\mathcal{C})/v)}$ is the pure
$\vert V\vert-d-2$ skeleton of the complex $\Delta$ on $V\setminus\{v\}$
with the single non-face $V\setminus(\{v\}\cup N)$. The facets of
$\Delta$ are a collection of codimension 1 faces of a simplex, hence
$\Delta$ \cite[proof of Proposition 8]{Eagon/Reiner:1998} and by
Lemma \ref{lem:VertDecSkeleton} $\alexdual{I(c_{d}(\mathcal{C})/v)}$
are vertex decomposable.\end{proof}
\begin{cor}
\emph{\label{cor:dComplementCompleteN_LinRes}(Emtander \cite[Theorem 4.1]{Emtander:2010})
}Let $\mathcal{C}$ be a $d$-uniform clutter such that every induced
subclutter has a complete-neighborhood vertex. Then the circuit ideal
of $c_{d}(\mathcal{C})$ has linear quotients, hence a linear resolution.
\end{cor}
Our complement operation $c_{d}$ outputs a $d$-uniform clutter,
even if the starting clutter was not uniform. If $\mathcal{C}$ is
not $d$-uniform, then $\alexdual{I(\mathcal{C})}$ is not pure, hence
not Cohen-Macaulay; but it could still be sequentially Cohen-Macaulay.
Algebraically, this corresponds with the circuit ideal being \emph{component-wise
linear} \cite{Herzog/Hibi:1999}. It might be interesting to find
an extension of Theorem \ref{thm:dComplementChordal_VD} involving
non-uniform clutters and component-wise linear circuit ideals.

\section{\label{sec:ForbiddenMinors}Forbidden minors}

\subsection{Obstructions to shellability}

Wachs defined an \emph{obstruction to shellability} to be a non-shellable
simplicial complex such that every induced subcomplex is shellable.
The obstructions to shellability that are \emph{flag complexes} (independence
complexes of graphs) were recently classified: Francisco and Van Tuyl
\cite{Francisco/VanTuyl:2007} showed that chordal graphs are sequentially
Cohen-Macaulay and that the $n$-cycle is an obstruction to shellability
for $n\neq3,5$. The author \cite{Woodroofe:2009a} showed that every
complex containing no such cycle is shellable. We see a close relationship
between the obstructions to shellability in flag complexes and the
forbidden subgraphs of a chordal graph.

It is easier to study obstructions to shellability in the special
case of flag complexes for at least two reasons. The first is that
graphs are better studied than clutters, and so there were pre-existing
theorems relating the forbidden subgraphs characterization of chordal
graphs to the simplicial vertex characterization. The second is that
every link in a flag complex can be expressed as an induced subgraph:
$\link_{I(G)}v=I(G\setminus N[v])$. We try to partially remedy the
latter with the following alternate definition:
\begin{defn}
A complex $\Delta$ is a \emph{$dc$-obstruction to shellability}
if $\Delta$ is non-shellable, but both every induced subcomplex and
every link are shellable. Here $dc$ stands for ``deletion-contraction''.
A non-shellable complex $\Delta$ such that $\link_{\Delta}v$ is
shellable for every $v\in V(\Delta)$ is a \emph{$c$-obstruction
to shellability}, and an obstruction to shellability in the sense
of Wachs is a \emph{$d$-obstruction to shellability}. \end{defn}
\begin{example}
The complex $\Delta$ with facets $\left\{ \{1,2,3\},\{3,4,5\},\{1,5\}\right\} $
is a $d$-obstruction but not a $dc$-obstruction to shellability,
since deleting any vertex leaves a connected complex with a single
$2$-face, but $\link_{\Delta}3$ is two disconnected edges, hence
not shellable. Similarly for the family constructed in \cite[Proposition 1]{Wachs:1999b}.
\end{example}
Since $d$-obstructions to shellability allow the possibility of complexes
where non-shellability is controlled by a proper (non-induced) subcomplex,
we regard the definition of $dc$-obstructions to shellability as
somewhat more natural. We comment that every pure $c$-obstruction
to shellability is a Buchsbaum complex, and that conversely a Buchsbaum
complex could be thought of as a pure ``$c$-obstruction to Cohen-Macaulay.'' 

Wachs conjectured \cite{Wachs:1999b} that there are a finite number
of $k$-dimensional $d$-obstructions to shellability for any fixed
$k$. Hachimori and Kashiwabara \cite[Theorem 4.6]{Hachimori/Kashiwabara:2011}
have recently shown that there are a finite number of $d$-obstructions
in dimension $k$ if and only if there are a finite number of $dc$-obstructions
in dimension $k$ --- in particular, replacing ``$d$-obstructions''
with ``$dc$-obstructions'' in the conjecture of Wachs leaves an
equivalent conjecture.

\subsection{Examples of forbidden minors}

A \emph{forbidden subclutter} of some family $\mathcal{F}$ of clutters
is a clutter $\mathcal{C}$ not in $\mathcal{F}$ such that every
induced subclutter is in $\mathcal{F}$. A \emph{forbidden minor}
of some family $\mathcal{F}$ of clutters is a clutter $\mathcal{C}$
that is not in $\mathcal{F}$, but such that every minor (obtained
by both deletion and contraction) is in $\mathcal{F}$. 

Every nonshellable forbidden subclutter to chordality is a $d$-obstruction
to shellability, while every nonshellable forbidden minor to chordality
is a $dc$-obstruction to shellability (since these two families are
the forbidden subclutters and forbidden minors to the family of clutters
with every subclutter shellable). Thus, an approach to the obstructions
to shellability problem is to understand the forbidden minors to chordal
clutters. 

For the case where $\mathcal{C}$ is a graph, we know that the forbidden
minors to chordality are exactly $C_{n}$ for $n\geq4$. The situation
with general clutters is open, but it seems quite reasonable to ask
whether every $dc$-obstruction to shellability is also a forbidden
minor to chordality. We present several examples of infinite families
of forbidden minors which are both. The hope is that a good understanding
of these forbidden minors could lead to deeper understanding (or even
a classification in the style of \cite{Woodroofe:2009a}) of obstructions
to shellability. 
\begin{example}
\label{exa:CyclicUniformClutter} Let $\mathcal{Z}_{n}^{k}$ be the
clutter with vertex set $\mathbb{Z}_{n}$ and circuits consisting
of every $k$ consecutive elements. Thus, $\mathcal{Z}_{n}^{2}\cong C_{n}$,
and more generally $\mathcal{Z}_{n}^{k}$ are the obvious $k$-uniform
extension of the cyclic graphs. Any vertex (hence every vertex) of
$\mathcal{Z}_{n}^{k}$ is simplicial if and only if $k=n$ or $n-1$,
so $\mathcal{Z}_{n}^{k}$ is not chordal unless $k=n$ or $n-1$.
Deleting any vertex leaves a clutter with the free vertex property,
so $\mathcal{Z}_{n}^{k}$ ($k\neq n,n-1$) are forbidden subclutters
to chordality. In some cases, for example $\mathcal{Z}_{5}^{3}$,
they may also be forbidden minors.
\end{example}
We take a brief detour to discuss some cases when $\mathcal{Z}_{n}^{k}$
is not a forbidden minor to chordality, i.e., when $\mathcal{Z}_{n}^{k}$
has a non-chordal contraction.
\begin{lem}
If $\ell k\leq n\leq\ell(k+1)$ and $k>2$, then $\mathcal{Z}_{n}^{k}$
has a contraction isomorphic to $\mathcal{Z}_{n-\ell}^{k-1}$.\end{lem}
\begin{proof}
The condition allows us to pick a set $S=\{v_{1},\dots,v_{\ell}\}$
vertices from $\mathbb{Z}_{n}=V(\mathcal{Z}_{n}^{k})$ so that every
2 vertices in $S$ have $k$ or $k-1$ vertices between them. Contracting
$S$ is easily seen to give $\mathcal{Z}_{n-\ell}^{k-1}$ as a minor.\qedhere

E.g., $\mathcal{Z}_{6}^{3}$ is not a forbidden minor to chordality,
since it contains a contraction minor isomorphic to the cyclic graph
$C_{4}$ ($=\mathcal{Z}_{4}^{2}$).
\end{proof}
More broadly, we could consider ``clutters of cyclic type'': clutters
on vertex set $\mathbb{Z}_{n}$ with all circuits consisting of consecutive
elements (possibly of different cardinalities). The next two examples,
however, show that not all forbidden minors have this form; moreover,
the results of Section \ref{sub:ComputationalResults} suggest that
such a form is relatively uncommon.

\begin{example}
\label{exa:DeletedCrossPolytope}Let $\mathcal{X}_{n}$ be the clutter
with vertex set $[2n]$ and circuits $\{\mbox{odd vertices}\}$, $\{\mbox{even vertices}\}$,
and $\{i,i+1\}$ for all odd $i$. By symmetry no vertex is simplicial,
and deleting or contracting any vertex leaves the same clutter (up
to isomorphism). Any such deletion or contraction removes one of the
two circuits with cardinality $n$, leaving a clutter with the free
vertex property. Thus $\mathcal{X}_{n}$ is a forbidden minor to chordality
for $n>1$.

The independence complex $I(\mathcal{X}_{n})$ is the boundary complex
of the $(n-1)$-dimensional cross-polytope with two opposing facets
removed, a non-shellable complex. Hence $I(\mathcal{X}_{n})$ is an
$dc$-obstruction to shellability for $n>1$.
\end{example}

\begin{example}
\label{exa:2Facets} Let $\mathcal{Y}_{n}$ be the clutter with vertex
set $[2n]$ and circuit set consisting of all $n$-sets except for
$\{1,\dots,n\}$ and $\{n+1,\dots,2n\}$. It is straightforward to
verify that every minor of $\mathcal{Y}_{n}$ is either a complete
uniform clutter, or else a complete uniform clutter with one circuit
removed. We notice that any vertex in the removed circuit of a latter
such minor is simplicial, hence every proper minor of $\mathcal{Y}_{n}$
is chordal.

The independence complex $I(\mathcal{Y}_{n})$ consists of all $(n-2)$
and lower dimensional faces, together with two disjoint $(n-1)$-faces.
As the pure $(n-1)$-skeleton is disconnected, we have for $n>1$
that $I(\mathcal{Y}_{n})$ is a $dc$-obstruction to shellability
and $\mathcal{Y}_{n}$ is a forbidden minor to chordality.
\end{example}
\begin{sidewaystable}
\caption{\label{tab:ObstructionsOn6Vertices}$dc$-obstructions to shellability
on 6 vertices}
\begin{longtable}{|l|>{\raggedright}p{3.9in}|>{\raggedright}p{2.6in}|l|}
\hline 
\noalign{\vskip\doublerulesep}
 &
Clutter of minimal non-faces &
Independence complex  &
Top skel.\tabularnewline[\doublerulesep]
\hline 
\noalign{\vskip\doublerulesep}
1.$\quad$ &
12, 13, 24, 35, 46, 56  &
145, 16, 236, 25, 34  &
$S^{0}$ \tabularnewline[\doublerulesep]
2. &
12, 13, 14, 235, 236, 245, 246, 256, 345, 346, 356, 456  &
156, 234, 25, 26, 35, 36, 45, 46  &
$S^{0}$ \tabularnewline[\doublerulesep]
3. &
12, 13, 14, 235, 236, 245, 256, 345, 356, 46  &
156, 234, 25, 26, 35, 36, 45  &
$S^{0}$ \tabularnewline[\doublerulesep]
4. &
12, 13, 14, 235, 236, 256, 356, 45, 46  &
156, 234, 25, 26, 35, 36  &
$S^{0}$ \tabularnewline[\doublerulesep]
5. &
12, 13, 14, 235, 246, 256, 36, 45  &
156, 234, 25, 26, 35, 46  &
$S^{0}$ \tabularnewline[\doublerulesep]
6. &
12, 13, 145, 146, 235, 236, 245, 246, 256, 345, 346, 356, \\
$\quad$456  &
14, 156, 234, 25, 26, 35, 36, 45, 46  &
$S^{0}$ \tabularnewline[\doublerulesep]
7. &
12, 13, 145, 146, 235, 245, 246, 256, 345, 36, 456  &
14, 156, 234, 25, 26, 35, 45, 46  &
$S^{0}$ \tabularnewline[\doublerulesep]
8. &
12, 13, 145, 146, 245, 26, 346, 35, 456  &
14, 156, 234, 25, 36, 45, 46  &
$S^{0}$ \tabularnewline[\doublerulesep]
9. &
12, 13, 145, 234, 236, 245, 246, 345, 346, 56  &
146, 15, 235, 24, 26, 34, 36, 45  &
$S^{0}$ \tabularnewline[\doublerulesep]
10. &
12, 13, 145, 236, 24, 345, 346, 56  &
146, 15, 235, 26, 34, 36, 45  &
$S^{0}$ \tabularnewline[\doublerulesep]
11. &
12, 13, 145, 24, 345, 36, 56  &
146, 15, 235, 26, 34, 45  &
$S^{0}$ \tabularnewline[\doublerulesep]
12. &
12, 13, 234, 235, 245, 345, 46, 56  &
145, 16, 236, 24, 25, 34, 35  &
$S^{0}$ \tabularnewline[\doublerulesep]
13. &
12, 134, 135, 136, 145, 146, 235, 236, 245, 246, 256, \\
$\quad$345, 346, 356, 456  &
13, 14, 156, 234, 25, 26, 35, 36, 45, \\
$\quad$46  &
$S^{0}$ \tabularnewline[\doublerulesep]
14. &
12, 134, 135, 136, 145, 234, 236, 245, 246, 345, 346, 56  &
13, 146, 15, 235, 24, 26, 34, 36, 45  &
$S^{0}$ \tabularnewline[\doublerulesep]
15. &
12, 134, 135, 146, 235, 246, 256, 36, 45  &
13, 14, 156, 234, 25, 26, 35, 46  &
$S^{0}$ \tabularnewline[\doublerulesep]
16. &
123, 124, 125, 126, 134, 135, 136, 145, 146, 235, 236,\\
$\quad$245, 246, 256, 345, 346, 356, 456  &
12, 13, 14, 156, 234, 25, 26, 35, 36, \\
$\quad$45, 46  &
$S^{0}$ \tabularnewline[\doublerulesep]
17. &
12, 134, 135, 146, 235, 246, 256, 345, 346, 356, 456  &
136, 145, 156, 234, 236, 245, 35, 46  &
$S^{1}$ \tabularnewline[\doublerulesep]
18. &
12, 134, 135, 234, 246, 345, 346, 56  &
136, 145, 146, 235, 236, 245, 34  &
$S^{1}$ \tabularnewline[\doublerulesep]
19. &
12, 134, 256, 35, 46  &
136, 145, 156, 234, 236, 245  &
$S^{1}$ \tabularnewline[\doublerulesep]
20. &
123, 124, 125, 126, 134, 135, 146, 235, 246, 256, 345, \\
$\quad$346, 356, 456  &
12, 136, 145, 156, 234, 236, 245, 35, \\
$\quad$46  &
$S^{1}$ \tabularnewline[\doublerulesep]
21. &
1234, 1235, 1246, 1356, 2456, 3456  &
1236, 1245, 1256, 1345, 1346, 1456, \\
$\quad$2345, 2346, 2356  &
$S^{2}$ \tabularnewline[\doublerulesep]
\hline 
\end{longtable}
\end{sidewaystable}

We note that of the 2-dimensional $dc$-obstructions to shellability
$M_{5}$, $M_{6}$, and $M_{7}$ that were considered by Wachs \cite[Lemma 5]{Wachs:1999b},
we have $M_{5}\cong I(\mathcal{Z}_{5}^{3})$, $M_{6}\cong I(\mathcal{X}_{3})$,
and $M_{7}\cong I(C_{7})=I(\mathcal{Z}_{7}^{2})$.

\subsection{\label{sub:ComputationalResults}Computational results}

Computation with GAP \cite{GAP4.4.12} yields exactly two forbidden
minors to chordality on $5$ vertices: the cyclic graph $C_{5}$ and
the clutter $\mathcal{Z}_{5}^{3}$ discussed in Example \ref{exa:CyclicUniformClutter}.
Both have homotopy type $S^{1}$. $I(\mathcal{Z}_{5}^{3})$ is a $dc$-obstruction
to shellability, while $I(C_{5})$ is shellable.

On 6 vertices, a similar computation yields 294 (isomorphism classes
of) forbidden minors to chordality on $6$ vertices (out of 16,353
non-isomorphic 6 vertex clutters). There are an additional 96 clutters
containing a $C_{5}$ minor (but no other non-chordal minor), all
96 of which are shellable. Of the 294 forbidden minors to chordality,
273 are shellable and $21$ are not. The shellable forbidden minors
to chordality on 6 vertices are too numerous to print here --- a complete
list and source code are available on my web page, currently at \url{http://www.math.wustl.edu/~russw}.

The $21$ non-shellable forbidden minors to chordality are the $dc$-ob\-structions
to shellability on $6$ vertices. These clutters and their independence
complexes are summarized in Table \ref{tab:ObstructionsOn6Vertices}.
The clutters and simplicial complexes are written in compact notation,
so that, for example, $12$ represents the set $\{1,2\}$. The fourth
column of Table \ref{tab:ObstructionsOn6Vertices} represents the
homotopy type of the pure top dimensional skeleton, as computed by
automatic collapsing of free faces. Since each is a sphere of lower
dimension than the top dimensional face, we see that none of these
complexes are sequentially Cohen-Macaulay.

We notice that the first 16 rows of the table represent simplicial
complexes consisting of two disjoint $2$-faces, with enough edges
between them to prevent a non-shellable minor. Line 1 represents the
cylic graph $C_{6}$ and its independence complex, and Line 16 is
isomorphic to the clutter $\mathcal{Y}_{3}$ discussed in Example
\ref{exa:2Facets}. Lines 17 to 20 represent simplicial complexes
consisting of an anulus formed by six $2$-facets, together with some
additional $1$-dimensional facets. Line 19 is isomorphic to the clutter
$\mathcal{X}_{3}$ discussed in Example \ref{exa:DeletedCrossPolytope}.
Line 21 is isomorphic to the clutter $\mathcal{Z}_{6}^{4}$ discussed
in Example \ref{exa:CyclicUniformClutter}, via the ordering of vertices
1, 2, 4, 6, 5, 3. Lines 1 and 21 are the only clutters of cyclic type.

The computation took several hours on a 2.4 Ghz MacBook, and involved
enumerating over all 6 vertex clutters. The main technical difficulty
was that computationally proving a complex to be non-shellable is
very slow. However, checking for a simplicial vertex and checking
for 4 or 5 vertex obstructions as minors are both fast, and give a
short list of complexes to check for shellability. Indeed, as the
non-shellable forbidden minors to chordality all had negative entries
in their $h$-triangle, it was only necessary to find shellings for
the shellable forbidden minors.
\begin{rem}
Hachimori and Kashiwabara \cite{Hachimori/Kashiwabara:2011} have
more recently used non-computational methods to classify all 2-dimensional
$d$-obstructions to shellability. 
\end{rem}
\bibliographystyle{hamsplain}
\bibliography{3_Users_russw_Documents_Research_Master}

\end{document}